\documentclass[11pt,reqno]{article}
\usepackage{amsthm,amsmath,amssymb,a4wide}
\usepackage{mathrsfs}
\usepackage{color}
\newtheorem{theorem}{Theorem}

\newtheorem{proposition}[theorem]{Proposition}
\newtheorem{lemma}[theorem]{Lemma}

\theoremstyle{remark}

\theoremstyle{definition}
\newtheorem{example}[theorem]{Example}

\numberwithin{theorem}{section}
\numberwithin{equation}{section}

\begin{document}
\title{ ``Overpartitionized'' Rogers-Ramanujan type identities}
\author{Abdulaziz Alanazi, Augustine O. Munagi, and Andrew V. Sills}
\date{\today}
\maketitle

\begin{abstract}
Many classical $q$-series identities, such as the Rogers--Ramanujan identities,
yield combinatorial interpretations in terms of integer
partitions.   Here we consider algebraically manipulating some of the classical $q$-series
to yield natural combinatorial interpretations in terms of overpartitions. 
Bijective proofs are supplied as well.
\end{abstract}

\section{Introduction}
A \emph{partition} $\lambda$ of an integer $n$ is a weakly decreasing finite sequence
$(\lambda_1, \lambda_2, \dots \lambda_\ell)$ of positive integers that sum to $n$.
Each $\lambda_i$ is called a \emph{part} of the partition $\lambda$.

Corteel and Lovejoy~\cite{CL03} introduced a variation on partitions, called \emph{overpartitions}, that has since inspired a massive literature.
An \emph{overpartition} of $n$ is a weakly decreasing finite sequence of positive integers
where the last occurrence of a part of a given magnitude may or may not be
distinguished by being overlined.   

For example, the partitions of three are $(3), (2,1), (1,1,1)$, while the overpartitions of three
are $(3), (\overline{3}), (2,1), (\overline{2},1), (2,\overline{1}), (\overline{2}, \overline{1}),
(1,1,1), (1,1,\overline{1})$.

There are many well known identities in the theory of partitions.
For example, if we
let $RR_1(n)$ denote the number of partitions of $n$ where all parts differ by at least $2$,
then the combinatorial version of the first Rogers--Ramanujan identity
(see, e.g., \cite{S17}) states that
$RR_1(n)$ equals the number of partitions of $n$ into parts congruent to $\pm 1\pmod{5}$.
In an effort to relate partition enumeration functions from classical partition identities to
enumeration functions for overpartitions, we prove below 
that $RR_1(n)$ also equals the number of overpartitions of $n$ in which the non-overlined parts are odd and distinct, and overlined
parts are at most the number of non-overlined parts.  We give overpartition interpretations
for various Rogers--Ramanujan type identities, including the G\"ollnitz--Gordon identities,
and the Little G\"ollnitz identities.

\section{Euler's Distinct Parts Partitions}
Perhaps the most famous partition identity of all time is Euler's theorem that states
that the number of partitions of $n$ into distinct parts equals the number of partitions
of $n$ into odd parts, see, e.g.,~\cite[p. 5, Cor. 1.2]{Andrews1976}
\begin{proposition} Let $D(n)$ denote the number of partitions of $n$ into distinct parts and let $\overline{D_k}(n)$, where $k\geq 1$, denote the number of overpartitions of $n$ in which the non-overlined parts are distinct and at least $k$ and the overlined parts are at most $k-1$. Then
    \begin{equation}
        D(n)=\overline{D_k}(n),
    \end{equation} for any fixed positive integer $k$.
\end{proposition}

\begin{proof}
    Fix a positive integer $k$.
    Let $D_k(n)$ denote the number of partitions of $n$ into distinct parts, with all parts
    at least
    $k$.  
    The generating function for $D_k(n)$ is 
    \begin{equation}\label{pro1}
    \sum_{n\geq 0} D_k(n) q^n = 
        \sum_{n\geq 0} \frac{q^{kn+(0+1+\cdots+(n-1))}}{(q;q)_n}=\sum_{n\geq 0} \frac{q^{kn}q^{\frac{n(n-1)}{2}}}{(q;q)_n}.  
       \end{equation} 
        Note that $D(n) = D_1(n)$, and in particular that 
      \begin{equation} \label{23} \sum_{n\geq 0} D(n) q^n = (-q;q)_\infty
    \end{equation}
Recall Euler's identity~\cite[p. 19, Eq. (2.2.6)]{Andrews1976} 
\begin{equation}\label{euler}
  (-aq;q)_\infty =   \sum_{n\geq0}\frac{a^nq^{\frac{n(n-1)}{2}}}{(q)_n}.
\end{equation}
By setting $a:=q^{k}$, and multiplying both side by $(-q;q)_{k-1}$, we get 
    \begin{equation} \label{25}
 (-q;q)_\infty  = \sum_{n\geq0}\frac{q^{kn}q^{\frac{n(n-1)}{2}}(-q;q)_{k-1}}{(q)_n}
  = \sum_{n\geq 0} \overline{D_k}(n) q^n .
\end{equation}
By combining~\eqref{23} and~\eqref{25}, the result follows.
\end{proof}

\begin{theorem}\label{D}
Let $\overline{E}(n)$ denote the number of overpartitions of $n$ in which the non-overlined parts are distinct with smallest part odd and the parity of 
adjacent parts alternate, where the largest overlined part is at most the number of non-overlined parts.  Then
   \begin{equation}
       D(n)=\overline{E}(n).
   \end{equation} 
\end{theorem}
\begin{proof}
   \begin{equation}
       \sum_{n=0}^\infty D(n)q^n=\sum_{n=0}^\infty \frac{ q^{n(n+1)/2} }{(q;q)_{n}}=\sum_{n=0}^\infty \frac{ q^{1+2+\cdots+n} (-q;q)_n }{(q^2;q^2)_{n}}=\sum_{n=0}^\infty \overline{E}(n)q^n
   \end{equation}
\end{proof}

The bijective proof will use the following lemma:
\begin{lemma}\label{lembit}
Let $B=(b_1,b_2,\ldots)$ be a nonempty binary sequence. There is a unique partition of least weight associated with $B$ which is gapfree with smallest part 1 (ignoring a possible initial string of 0's).
\end{lemma}
\begin{proof}
The partition, denoted by $t(B)=(t_1,t_2,\ldots)$, may be obtained as follows.
Set $t_1=b_1$, then for any index $j>1$, 
\begin{equation}\label{eqbit}
t_j =\begin{cases} t_{j-1}& \text{if}\ b_j=b_{j-1}\\
t_{j-1}+1 & \text{if}\ b_j\neq b_{j-1}.
\end{cases}
\end{equation}
For example, if $B=(0,1,1,0,1,0,0)$ then using \eqref{eqbit} we obtain $t(B)=(0,1,1,2,3,4,4)$.
\end{proof}

\begin{proof}[Bijective Proof of Theorem \ref{D}] 
Let $\overline{E}[n]$ and $D[n]$ denote the corresponding sets of overpartitions and ordinary partitions of $n$ respectively.
We describe a map $f: D[n]\rightarrow \overline{E}[n]$.
Let $\lambda\in D[n]$ where $\lambda = (\lambda_1,\lambda_2,\ldots,\lambda_k)$ with $1\leq \lambda_1<\lambda_2<\cdots <\lambda_k$.
Then
\begin{equation}\label{eqmap}
f(\lambda)=\begin{cases} \lambda& \text{if $\lambda_1$ is odd and $\lambda_i\not\equiv \lambda_{i+1}$ (mod 2)}\ \forall\, i\\
\gamma & \text{otherwise},
\end{cases}
\end{equation}
where $\gamma$ is an overpartition of $n$ obtained as follows. Let $B_k$ denote the parity-alternating binary sequence of length $k$ with first term 1, $B_k=(1,0,1,0,\ldots)$.
\begin{itemize}
\item[(i)] Set $A :=(\lambda-B_k)$ (mod 2);
\item[(ii)] Use Lemma \ref{lembit} to obtain $t(A)$ and denote its conjugate by $t(A)'$;
\item[(iii)] Set $\gamma$ to be the overpartition whose non-overlined parts consist of $\lambda-t(A)$ and whose overlined parts consist of $t(A)'$.  
\end{itemize}

Conversely, let $\gamma\in \overline{E}[n]$. If $\gamma$ has no overlined parts then $f^{-1}(\gamma)=\gamma$. Otherwise denote the sets of the non-overlined and overlined parts of $\gamma$ by $U$ and $V$ respectively. Then $f^{-1}(\gamma) = U+V'$, where $V'$ may be padded with initial 0's if necessary.

It is clear that $A\equiv t(A)$ (mod 2) termwise, by construction. Thus if $\gamma=(\gamma_1,\ldots,\ldots,\gamma_k)$ and $t(A)=(t_1,\ldots,\ldots,t_k)$, then for any index $j$, we have $\gamma_j\equiv \lambda_j-t_j\equiv \lambda_j-(\lambda_j-(j\, \text{mod 2}))\equiv j$ (mod 2). Since $t_k\leq k$ and $t(A)$ is gapfree with 1, the conjugate $t(A)'$ has distinct parts with largest part at most $k$. Thus in step (iii), $\gamma\in \overline{E}[n]$. So $f$ is a bijection. 

\vskip 5pt
For example, consider $\lambda=(1, 2, 4, 5, 13, 14)\in D[39]$. Then $A=\lambda-(1,0,1,0,1,0)\equiv (0,0,1,1,0,0)$ (mod 2). So $t(A)=(0,0,1,1,2,2)$, $t(A)'=(1,1,2,2)'=(2,4)$ and $\lambda-t(A) = (1, 2, 3, 4, 11, 12)$. Hence $f(\lambda)=(1, 2, 3, 4, 11, 12,\overline{2},\overline{4})\in \overline{E}[39]$.

\end{proof}

\section{The Rogers--Ramanujan Identities}
Let $RR_1(n)$ be the number of partitions of $n$ in which the difference between parts is at least $2$.
The partition theoretic form of the first Rogers--Ramanujan identity states that $RR_1(n)$
equals the number of partitions of $n$ into parts congruent to $1$ or $4$ modulo 
$5$~\cite[p. 109, Cor. 7.6]{Andrews1976}, for all integers $n$.
\begin{theorem}\label{FRR}
Let $\overline{RR_1}(n)$ be the number of overpartitions of $n$ in which the non-overlined parts are odd and distinct and overlined parts are at most the number of non-overlined parts. Then 
      $ RR_1(n)=\overline{RR_1}(n).$
\end{theorem}
\begin{proof}[Generating function proof]
   \begin{equation}
       \sum_{n=0}^\infty RR_1(n)q^n=\sum_{n=0}^\infty \frac{ q^{n^2} }{(q;q)_{n}}=\sum_{n=0}^\infty \frac{ q^{n^2} (-q;q)_n }{(q^2;q^2)_{n}}=\sum_{n=0}^\infty \overline{RR_1}(n)q^n
   \end{equation}
\end{proof}
\begin{proof}[Bijective proof]
Let $RR_1[n] $ and $ \overline{RR_1}[n]$ denote the sets of partitions enumerated by $RR_1(n) $ and $\overline{RR_1}(n)$. We give a map $h: RR_1[n]\rightarrow \overline{RR_1}[n]$. Let $\lambda=(\lambda_1,\ldots,\lambda_k)\in RR_1[n]$, where $\lambda_1>\ldots>\lambda_k$, and assume $h(\lambda)=\pi\in \overline{RR_1}[n]$ such that  $\pi=(\pi_1,\ldots,\pi_k,\overline{v_1},\ldots,\overline{v_r})\in \overline{RR_1}[n]$ with $\pi_1>\cdots>\pi_k$ and $\overline{v_1}>\cdots>\overline{v_r}$. 

Then for each $\lambda_j$, define $\ell^{oe}_j$ as the number of occurrences of pairs of adjacent parts $\lambda_i,\lambda_{i+1}$ that are odd and even respectively, such that $\lambda_j\geq \lambda_i$. Then
\begin{equation}\label{eqoe}
\pi_j  := \begin{cases}
  \lambda_j-2\ell^{oe}_j & \mbox{ if $ \text{ $\lambda_j$ is odd} $}\\
  \lambda_j-2\ell^{oe}_j-1 & \mbox{ if \text{ $\lambda_j$ is even}}.
  \end{cases}
\end{equation}
Now define the partition $v^*=(v^*_1,\ldots,v^*_k)$, where $v^*_j=\lambda_j-\pi_j$ for all $j$. The overlined parts of $\pi$ consist of the conjugate $(v^*)'$, say  $(v_1,\ldots,v_r)$. 
\medskip
  
Conversely, let $\pi\in \overline{RR_1}[n]$ and denote the sets of the non-overlined and overlined parts of $\pi$ by $U$ and $V$ respectively. Then $h^{-1}(\pi) = U+V'$ (cf. $f^{-1}$ in the proof of Theorem \ref{D}).

Since $\lambda_j-\lambda_{j+1}>1$, the definition \eqref{eqoe} insures that non-overlined parts of $\pi$ are odd with $\pi_j-\pi_{j+1}\geq 2$ for all $j$. Hence the map $h$ is well-defined.
\end{proof}

\begin{example}
Let $\lambda=(20,18,15,13,10,7,4,1)\in RR_1[88]$. Then $\ell^{oe}_1=2, \ell^{oe}_2=2, \ell^{oe}_3=2, \ell^{oe}_4=2, \ell^{oe}_5=1,\ell^{oe}_6=1,\ell^{oe}_7=0$ and $\ell^{oe}_8=0$. So the non-overlined parts of $h(\lambda)=\pi$ are $(20-2(2)-1,18-2(2)-1,\ldots ) = (15,13,11,9,7,5,3,1)$. Thus $v^*=(5,5,4,4,3,2,1,0)$ with $(v^*)'=(7,6,5,4,2)$. Therefore,  $\pi=(15,13,11,9,7,5,3,1,\overline{7},\overline{6},\overline{5},\overline{4},\overline{2}).$
\end{example}
 
\begin{theorem}\label{FRR2}
Let $\overline{RR^*_1}(n)$ denote the number of overpartitions of $n$ in which non-overlined parts are even and distinct and overlined parts are at most one plus the number of non-overlined parts. Then
     \begin{equation}
       RR_1(n)=\overline{RR^*_1}(n)
   \end{equation}  
  \end{theorem}
\begin{proof}[Generating function proof]
It is clear that the generating function for the number of partitions of $n$ in which the difference between parts is at least $2$ can be separated into two terms: the first enumerating 
such partitions that contain a $1$ as part and the second enumerating such partitions
that do not contain a $1$ as a part.


Therefore,
   \begin{align}
   \sum_{n=0}^\infty RR_1(n)q^n=\sum_{n=0}^\infty \frac{ q^{n^2+n} }{(q;q)_{n}}+\sum_{n=0}^\infty \frac{ q^{n^2+2n+1} }{(q;q)_{n}}&=\sum_{n=0}^\infty \frac{ q^{n^2+n} }{(q;q)_{n}}(1+q^{n+1})\\ &=\sum_{n=0}^\infty \frac{ q^{n^2+n} (-q;q)_n}{(q^2;q^2)_{n}}(1+q^{n+1})\\ &=\sum_{n=0}^\infty \frac{ q^{n^2+n} (-q;q)_{n+1}}{(q^2;q^2)_{n}}=\sum_{n=0}^\infty \overline{RR^*_1}(n)q^n
\end{align} 
\end{proof}

\begin{proof}[Bijective Proof]
This may be proved using the bijection $h$ as in the proof of Theorem \ref{FRR}. We only need to replace odd-even pairs with even-odd pairs of adjacent parts, and correspondingly define
\[ \pi_j  := \begin{cases}
  \lambda_j-2\ell^{eo}_j & \mbox{ if $ \text{ $\lambda_j$ is even} $}\\
  \lambda_j-2\ell^{eo}_j-1 & \mbox{ if \text{ $\lambda_j$ is odd}}.
  \end{cases}
  \]
\end{proof}

\begin{example}
 Let $\lambda=(20,18,15,13,10,7,4,1)\in RR_1[88]$. Then $\ell^{eo}_1=3, \ell^{eo}_2=3, \ell^{eo}_3=2, \ell^{eo}_4=2,  \ell^{eo}_5=2,\ell^{eo}_6=1,\ell^{eo}_7=1$ and $\ell^{oe}_8=0$. So the non-overlined parts are of $\pi$ are $(20-2(3),18-2(3),\ldots ) = (14,12,10,8,6,4,2,0)$. So $v^*=(6,6,5,5,4,3,2,1)$ with $(v^*)'=(8,7,6,5,4,2)$. Hence  $\pi=(14,12,10,8,6,4,2,\overline{8},\overline{7},\overline{6},\overline{5},\overline{4},\overline{2}).$
\end{example}

Let $RR_2(n)$ denotes the number of partitions of $n$ into parts $>1$ and the difference between parts is at least $2$.  The second Rogers--Ramanujan identity states that
for all integers $n$,
$RR_2(n)$ equals the number of partitions of $n$ into parts congruent to $2$ or $3$
modulo $5$~\cite[p. 109, Cor. 7.7]{Andrews1976}.
\begin{theorem}\label{SRR}
 Let $\overline{RR_2}(n)$ denote the number of overpartitions of $n$ in which the non-overlined parts are even and distinct and overlined parts are at most 
 the number of non-overlined parts. Then 
      $ RR_2(n)=\overline{RR_2}(n).$
\end{theorem}

\begin{proof}
   \begin{equation}
       \sum_{n=0}^\infty RR_2(n)q^n=\sum_{n=0}^\infty \frac{ q^{n^2+n} }{(q;q)_{n}}=\sum_{n=0}^\infty \frac{ q^{n^2+n} (-q;q)_n }{(q^2;q^2)_{n}}=\sum_{n=0}^\infty \overline{RR_2}(n)q^n
   \end{equation}
\end{proof}

The same bijection used in the proof of Theorem \ref{FRR2} will suffice to give a bijective proof of Theorem~\ref{SRR}.

\begin{example}
    Let $\lambda=(20,18,15,13,10,7,4)\in RR_2[87]$. Then $\ell^{eo}_1=2, \ell^{eo}_2=2, \ell^{eo}_3=1, \ell^{eo}_4=1, \ell^{eo}_5=1,\ell^{eo}_6=0$ and $\ell^{oe}_7=0$. The non-overlined parts of the image $\pi$ are $(16,14,12,10,8,6,4)$, and $v^*=(4,4,3,3,2,1,0)$ with $(v^*)' = (6,5,4,2)$. Thus\\ $\pi=(16,14,12,10,8,6,4,\overline{6},\overline{5},\overline{4},\overline{2}).$
\end{example}

We next apply a similar reasoning to the one used in the generating function proof of Theorem \ref{FRR2}. 

We write the generating function for partitions into parts greater than $1$ and differing by
at least two as a difference between those partitions into parts differing by at least two 
and those partitions differing by at least two and containing a $1$:

\begin{align}
   \sum_{n=0}^\infty \frac{ q^{n^2} }{(q;q)_{n}}-\sum_{n=0}^\infty \frac{ q^{n^2+2n+1} }{(q;q)_{n}}&=\sum_{n=0}^\infty \frac{ q^{n^2} }{(q;q)_{n}}(1-q^{2n+1})\\ &=\sum_{n=0}^\infty \frac{ q^{n^2} (-q;q)_n}{(q^2;q^2)_{n}}(1-q^{n+1})\\ &=\sum_{n=0}^\infty \frac{ q^{n^2} (-q;q)_{n} (q;q^2)_{n+1}}{(q;q)_{2n}}.
\end{align}

The terms of the series
$$\sum_{n=0}^\infty \frac{ q^{n^2}  (q;q^2)_{n+1}}{(q;q)_{2n}}$$ 
are tabulated in~\cite[A027349]{S1} under the description ``The number of partitions of $n+1$ into distinct odd parts, the least being $1$."\\
Notice that we might write the generating function for the number of partitions of $n$ into distinct odd parts, the least being $1$ as follows:
\begin{align}
   \sum_{n=0}^\infty \frac{ q^{n^2+2n+1} }{(q^2;q^2)_{n}}=\sum_{n=0}^\infty \frac{ q^{n^2+2n+1} (q;q^2)_n }{(q;q)_{2n}}.
\end{align}

Thus $$\sum_{n=0}^\infty \frac{ q^{n^2+2n} (q;q^2)_n }{(q;q)_{2n}}=\sum_{n=0}^\infty \frac{ q^{n^2}  (q;q^2)_{n+1}}{(q;q)_{2n}}.$$

\section{The G\"ollnitz--Gordon Identities}
Let $GG_1(n)$ denote the number of partitions of $n$ where parts mutually differ by $2$ and no consecutive even numbers appear as parts.
The first G\"ollnitz--Gordon identity states that for all integers $n$, $GG_1(n)$ equals the
number of partitions into parts congruent to 1, 4, or 7 modulo 8.~\cite[p. 80, Thm. 2.42]{S17}.

\begin{theorem}\label{FGG}
Let $\overline{GG_1}(n)$ denote the number of overpartions of $n$ where the non-overlined parts are odd and distinct and overlined parts are odd and at most $2$ times the number non-overlined parts minus $1$. Then
     $  GG_1(n)=\overline{GG_1}(n).$
\end{theorem}

The generating function proof is straightforward as the generating function may be interpreted in two different ways.   For a detailed explanation of the first equality
in~\eqref{gg1} below, see~\cite[pp. 80--83]{S17}.
     \begin{equation} \label{gg1}
       \sum_{n=0}^\infty GG_1(n)q^n=\sum_{n=0}^\infty \frac{ q^{n^2}(-q;q^2)_n }{(q^2;q^2)_{n}}=\sum_{n=0}^\infty \overline{GG_1}(n)q^n
   \end{equation}
   
\begin{proof}[Bijective proof]
Let $GG_1[n]$ and and $\overline{GG_1}[n]$ denote the sets of partitions enumerated by $GG_1(n)$ and $\overline{GG_1}(n)$. We will define a map $\, g:GG_1[n]\rightarrow \overline{GG_1}[n]$. Let $\lambda=(\lambda_1,\lambda_2,\cdots, \lambda_r)\in GG_1[n]$, where $\lambda_1>\cdots >\lambda_r$, and define the parity function 
 \[ T(k)  := \begin{cases}
  0 & \mbox{ if $ \text{ $k$ is odd} $}\\
  1 & \mbox{ if \text{ $k$ is even}}.
  \end{cases}
  \]
  Therefore,
$$(\lambda_1,\cdots, \lambda_r)\mapsto (\tau_1,\cdots, \tau_r,\overline{T(\lambda_1)\cdot 1},\overline{T(\lambda_2)\cdot 3},\cdots , \overline{T(\lambda_r)\cdot (2r-1)})\in \overline{GG_1}[n],$$
where 
\begin{equation}\label{nonoldef}
\tau_j=\lambda_j-\left(T(\lambda_j)+2\sum_{i=j+1}^r T(\lambda_i)\right).
\end{equation}

The inverse map $g^{-1}:\overline{GG_1}[n]\rightarrow GG_1[n] $ runs as follows. 
Any overpartition $\tau\in \overline{GG_1}[n]$ can be written in the form $\tau=(\tau_1,\tau_2,\cdots, \tau_r,\overline{v_1\cdot 1},\overline{v_2\cdot 3},\cdots , \overline{v_r\cdot (2r-1)})$, where $v_k=1$ if $\overline{(2k-1)}\in \tau$ and $v_k=0$ if $\overline{(2k-1)}\notin \tau$. Thus we define 
$$\tau=(\tau_1,\tau_2,\cdots, \tau_r,\overline{v_1\cdot 1},\overline{v_2\cdot 3},\cdots , \overline{v_r\cdot (2r-1)})\rightarrow (\lambda_1,\lambda_2,\cdots, \lambda_r)$$
where 
$$\lambda_j=\tau_j+v_j+2\sum_{i=j+1}^r v_i,\ 1\leq j\leq r.$$

We remark that if $\lambda$ contains two consecutive even parts, say $\lambda_j=2s$ and $\lambda_{j+1}=2s-2$, then using \eqref{nonoldef} we have 
\begin{align*}
\tau_j & = \lambda_j-T(\lambda_j)-2T(\lambda_{j+1})-2(T(\lambda_{j+2}+\cdots+\lambda_r)\\  
& = (2s)-(1)-2(1)-2(T(\lambda_{j+2}+\cdots+\lambda_r)\\
&=\tau_{j+1}.
\end{align*}
That is, the non-overlined parts of $\tau$ will not be distinct. So the map $g$ is well defined.
\end{proof}

\begin{example}
    Let $\lambda=(20,17,15,12,9,7,4,1)\in GG_1[85]$ then $\tau_1=20-(1=2(2))=15, \tau_2=17-(0+2(2)), \tau_3=15-(0+2(2))=13, \tau_4=12-(1+2(1))=9, \tau_5=9-(0+2(1))=7,\tau_6=7-(0+2(1))=5,\tau_7=4-(1+2(0))=3$ and $\tau_8=1-(0+2(0))=1$. So the non-overlined parts are $(15,13,11,9,7,5,3,1)$ and, since the indices of even parts in $\lambda$ are $1,4,7$, the overlined parts are $(1,7,13)$. Thus $\tau=(15,13,11,9,7,5,3,1,\overline{1},\overline{7},\overline{13}).$
\end{example}
\medskip

Let $GG_2(n)$ be the number of partitions of $n$ in which parts mutually differ by $2$, no part less than $3$ and no consecutive even numbers appear as parts.
The second G\"ollnitz--Gordon identity states that for all integers $n$, $GG_1(n)$ equals
the number of partitions of $n$ into parts congruent to 3, 4, or 5 modulo
8~\cite[p. 80, Thm. 2.43]{S17}. 

\begin{theorem}\label{SGG}
Let $\overline{GG_2}(n)$ denote the number of overpartions of $n$ where the non-overlined parts are odd, distinct and greater than $1$ and overlined parts are odd and at most $2$ times the number non-overlined parts minus $1$. Then
      $ GG_2(n)=\overline{GG_2}(n).$
\end{theorem}

\begin{proof}[Generating function proof]
   The generating function proof is straightforward as one generating function may be interpreted in two different ways:
     \begin{equation}
       \sum_{n=0}^\infty GG_2(n)q^n=\sum_{n=0}^\infty \frac{ q^{n^2+2n}(-q;q^2)_n }{(q^2;q^2)_{n}}=\sum_{n=0}^\infty \overline{GG_2}(n)q^n.
   \end{equation}
  \end{proof}
  
  \begin{proof}[Bijective proof]
   The bijective map $g$ in the proof of Theorem \ref{FGG} may be used to establish this theorem as well.
\end{proof}
\begin{example}
    Let $\lambda=(20,17,15,12,9,7,4)\in GG_1[84]$ then $\tau_1=20-(1=2(2))=15, \tau_2=17-(0+2(2)), \tau_3=15-(0+2(2))=13, \tau_4=12-(1+2(1))=9, \tau_5=9-(0+2(1))=7,\tau_6=7-(0+2(1))=5$ and $\tau_7=4-(1+2(0))=3$. So the non-overlined parts are $(15,13,11,9,7,5,3)$, and the overlined parts are $(1,7,13)$. Thus $\tau=(15,13,11,9,7,5,3,\overline{1},\overline{7},\overline{13}).$
\end{example}
\medskip

\begin{theorem}\label{DGG} Let $DGG_{1,2}(n)$ denote the number of partitions of $n$ into parts that mutually differ by $2$, no consecutive even numbers appear as parts such that one part equals $1$ or $2$. Let $\overline{DGG_{1,2}}(n)$ be the number of overpartitions of $n$ when non-overlined parts are distinct, odd and include $1$, and the overlined parts are odd and at most $2$ times the number of non-overlined parts minus $1$.
Then
    \begin{equation}
       DGG_{1,2}(n)=\overline{DGG_{1,2}}(n).
   \end{equation} 
\end{theorem}
\begin{proof}
    \begin{align} \sum_{n=0}^\infty DGG_{1,2}(n) q^n=\sum_{n=0}^\infty \frac{q^{n^2} (-q;q^2)_n}{(q^2;q^2)_n} &- \sum_{n=0}^\infty \frac{q^{n^2+2n} (-q;q^2)_n}{(q^2;q^2)_n} \\ &= \sum_{n=1}^\infty \frac{q^{n^2}(-q;q^2)_n}{(q^2;q^2)_n}( 1 -q^{2n} )\\ &= \sum_{n=1}^\infty \frac{q^{n^2}(-q;q^2)_n}{(q^2;q^2)_{n-1}}\\ 
&= \sum_{n=1}^\infty \overline{DGG_{1,2}}(n) q^n.
\end{align}
For the bijective proof, again the map $g$ would work for this identity. 
\end{proof}

\section{The Little G\"ollnitz Identities}

Let $LG_1(n)$ denotes the number of partitions of $n$ where parts that mutually differ by $2$ and in which no consecutive odd numbers appear as parts.
The first little G\"ollnitz identity states that for all integers $n$, 
$LG_1(n)$ equals the number of partitions of $n$ into distinct parts congruent to 0, 1, or 2
modulo 4 and also equal to the number of partitions of $n$ into parts congruent to
1, 5, or 6 modulo 8~\cite[p. 83, Thm 2.44]{S17}. 

\begin{theorem}\label{LG}
Let $\overline{LG_1}(n)$ denote the number of overpartitions of $n$ in which the 
non-overlined parts are even and distinct and the overlined parts are odd and at most one more than two times the number of non-overlined parts. Then
      $ LG_1(n)=\overline{LG_1}(n)$.
\end{theorem}

\begin{proof}
We know that 
\begin{equation}\label{LG1}
    \sum_{n=0}^{\infty}LG_1(n)q^n=\sum_{n=0}^\infty \frac{ q^{n(n+1)} (-q^{-1};q^2)_{n}}{(q^2;q^2)_n}=(-q^2;q^2)_\infty(-q;q^4)_\infty
\end{equation}
   
Recall V.-A. Lebesgue's identity~\cite{L40}
\begin{equation} \label{L1}
L(a;q):=\sum_{n=0}^\infty \frac{ q^{\frac{n(n+1)}{2}} (a;q)_n}{(q;q)_n} = (-q; q )_\infty (aq;q^2)_\infty.
\end{equation}
By replacing $q$ by $q^2$, setting $a:=-q$ and multiplying both sides by $(1+q)$ in \eqref{L1}, then 
\begin{equation} \label{L2}
L(-q;q^2)=\sum_{n=0}^\infty \frac{ q^{n(n+1)} (-q;q^2)_{n+1}}{(q^2;q^2)_n} = (-q^2; q^2 )_\infty (-q;q^4)_\infty.
\end{equation}
Therefore, for \eqref{LG2} and \eqref{L2}, we deduce that 
\begin{equation}\label{L3}
   \sum_{n=0}^{\infty}LG_1(n)q^n=\sum_{n=0}^\infty \frac{ q^{n(n+1)} (-q^{-1};q^2)_{n}}{(q^2;q^2)_n}= \sum_{n=0}^\infty \frac{ q^{n(n+1)} (-q;q^2)_{n+1}}{(q^2;q^2)_n} =\sum_{n=0}^{\infty}\overline{LG_1}(n)q^n.
\end{equation}
For the bijective proof, If the parity function in the  previous bijection is changed as follows:\\
\[ T(k)  := \begin{cases}
  0 & \mbox{ if $ \text{ $k$ is even} $}\\
  1 & \mbox{ if \text{ $k$ is odd}},
  \end{cases}
  \]
 then the defined map $g$ in Theorem \ref{FGG} provides a bijective proof of this theorem as well.
\end{proof}
\medskip

Let $LG_2(n)$ denote the number of partitions of $n$ where parts mutually differ by $2$ with no $1$'s such that no consecutive odd numbers appear as parts.
The second little G\"ollnitz identity states that $LG_2(n)$ equals the number of partitions into distinct parts congruent to 0, 2, or 3 modulo 4, and also 
equal to the number of partitions of $n$ into parts congruent to 2, 3, or 7 modulo 8~\cite[p. 83, Thm. 2.45]{S17}.

\begin{theorem}\label{SLG}
Let $\overline{LG_2}(n)$ denote the number of overpartitions of $n$ where the non-overlined parts are even and distinct, and the overlined parts are odd and at most $2$ times the number of non-overlined parts. Then
      \begin{equation}\label{LG2}
       LG_2(n)=\overline{LG_2}(n).
   \end{equation} 
\end{theorem}

\begin{proof}[Generating function proof]
   The generating function proof is straightforward as the generating function may be interpreted in two different ways. 
     \begin{equation}
       \sum_{n=0}^\infty LG_2(n)q^n=\sum_{n=0}^\infty \frac{ q^{n^2+n}(-q;q^2)_n }{(q^2;q^2)_{n}}=\sum_{n=0}^\infty \overline{LG_2}(n)q^n.
   \end{equation}
   \end{proof}

\begin{proof}[Bijective proof] The same bijective proof used for Theorem \ref{LG} applies here as well. 
\end{proof}

\subsection{General Case}
Recall again V.-A. Lebesgue's identity~\eqref{L1}.
Let $k=4\alpha+\beta$, where $\alpha \geq 0$ and $\beta \in \{-1,0,1,2\}$, then by replacing $q$ by $q^2$, setting $a:=-q^k$ and multiplying both sides in \eqref{L1} by $(-q^{\beta+2};q^4)_\alpha$, we get 
\begin{equation} \label{L2}
L(-q^k;q^2)=\sum_{n=0}^\infty \frac{ q^{n(n+1)} (-q^k;q^2)_n (-q^{\beta+2};q^4)_\alpha}{(q^2;q^2)_n} = (-q^2; q^2 )_\infty (-q^{\beta+2};q^4)_\infty.
\end{equation}

\begin{equation} 
L(-q^k;q^2)=\sum_{n=0}^\infty \frac{ q^{n(n+1)} (-q^{k-2};q^2)_{n+1} (-q^{\beta+2};q^4)_{\alpha-1}}{(q^2;q^2)_n} = (-q^2; q^2 )_\infty (-q^{\beta+2};q^4)_\infty.
\end{equation}
The left hand side of ~\eqref{L2} may be interpreted (except when $\alpha=0$ and $\beta=0$) as the number of overpartitions of $n$ in which\\
\begin{enumerate}
    \item the non-overlined parts are even and distinct,
    \item overlined parts $\geq k$ are odd (even when $\beta$ is even) and at most $2\times$(the number of non-overlined parts)$+(k-2)$
    \item overlined parts $<k$ are $\equiv (\beta+2) \pmod 4$.
\end{enumerate}
In the case when $k=0$, $$\sum_{n=0}^\infty \frac{ q^{n(n+1)} (-1;q^2)_{n} }{(q^2;q^2)_n}=2\sum_{n=0}^\infty \frac{ q^{n(n+1)} (-q^2;q^2)_{n-1} }{(q^2;q^2)_n}$$
which will be interpreted as twice
the generating function of the number of overpartitions of $n$ in which the non-overlined parts 
are even and distinct and overlined parts are even and at most two less than twice the number 
of non-overlined parts. Note that, the right hand side possesses an overpartition interpretation 
if we allow $\overline{0}$ as a possible overlined part. 

The right hand side of~\eqref{L2} 
\[ L(-q^k;q^2)  = \begin{cases}
  (q^1,q^5,q^6;q^8)^{-1}_\infty & \mbox{ if \text{ $\beta=-1$}}\\
 (-q^2;q^4)_\infty(q^2;q^4)^{-1}_\infty & \mbox{ if $ \text{ $\beta=0$ } $}\\
  (q^2,q^3,q^7;q^8)^{-1}_\infty & \mbox{ if $ \text{ $\beta=1$ } $}\\
   (q^2,q^4,q^6;q^8)^{-1}_\infty & \mbox{ if $ \text{ $\beta=2$ } $}.
  \end{cases}
  \]

  Let $\alpha=0$ and $\beta=-1$ in~\eqref{L2} then we get~\eqref{LG2}. By letting $\alpha=0$ and $\beta=1$ in~\eqref{L2} then we get~\eqref{LG1}. 

In~\cite{SS11}, Savage and Sills derived new interpretations of little G\"ollnitz identities by utilizing the following specialization of the $q$-Gauss identity~\cite[p. 14, Eq. (1.5.1)]{GR04}:
\begin{equation} \label{HG1}
H(a,c;q):=\sum_{n=0}^\infty \frac{ q^{\frac{n(n-1)}{2}}(\frac{-c}{a})^n (a;q)_n}{(c;q)_n(q;q)_n} = \left(\frac{c}{a}; q \right)_\infty (c;q)^{-1}_\infty.
\end{equation}
 They showed that 
 \begin{equation} \label{HGL1}
H(-q, q^2;q^4)=\sum_{n=0}^\infty \frac{ q^{2n^2-n} (-q;q^4)_n}{(q^2;q^2)_{2n}} =  (q^1,q^5,q^6;q^8)^{-1}_\infty
\end{equation}
and 
\begin{equation} \label{HGL2}
H(-q^{-1},q^2;q^4)=\sum_{n=0}^\infty \frac{ q^{2n^2+n} (-q^{-1};q^4)_n}{(q^2;q^2)_{2n}} =  (q^2,q^3,q^7;q^8)^{-1}_\infty.
\end{equation}

By using the same manipulations, we can deduce that
\begin{equation} \label{HGL3}
H(-1,q^2;q^4)=\sum_{n=0}^\infty \frac{ q^{2n^2} (-1;q^4)_n}{(q^2;q^2)_{2n}} = (-q^2;q^4)_\infty(q^2;q^4)^{-1}_\infty
\end{equation}
and 
\begin{equation} \label{HGL4}
H(-q^{-2}, q^2;q^4)=\sum_{n=0}^\infty \frac{ q^{2n^2+2n} (-q^{-2};q^4)_n}{(q^2;q^2)_{2n}} =  (q^2,q^4,q^6;q^8)^{-1}_\infty.
\end{equation}

Therefore, by~\eqref{L2} when ($\alpha\geq0$ and $\beta=0$) and~\eqref{HGL3}
\begin{equation} \label{HGLL3}
\sum_{n=0}^\infty \frac{ q^{2n^2} (-1;q^4)_n}{(q^2;q^2)_{2n}} = \sum_{n=0}^\infty \frac{ q^{n(n+1)} (-q^{4\alpha};q^2)_n (-q^{4};q^4)_\alpha}{(q^2;q^2)_n},
\end{equation}
and by~\eqref{L2} when ($\alpha\geq0$ and $\beta=2$) and~\eqref{HGL4}
\begin{equation} \label{HGLL4}
\sum_{n=0}^\infty \frac{ q^{2n^2+2n} (-q^{-2};q^4)_n}{(q^2;q^2)_{2n}} = \sum_{n=0}^\infty \frac{ q^{n(n+1)} (-q^{4\alpha+2};q^2)_n (-q^{4};q^4)_\alpha}{(q^2;q^2)_n}.
\end{equation}

We can obtain another identity 
\begin{equation} \label{HGL5}
H(-q^{2},q^4;q^4)=\sum_{n=0}^\infty \frac{ q^{2n^2} (-q^{2};q^4)_n}{(q^4;q^4)^2_{n}} =  (q^2,q^6,q^8;q^8)^{-1}_\infty.
\end{equation}

Recall Slater~\cite[Eq. (47)]{S52} 

\begin{equation} \label{Sl47}
\sum_{n=0}^\infty \frac{ q^{n^2} (-1;q^2)_n}{(q;q)_{2n}} =\frac{(-q;q^2)_\infty(q^2,q^6;q^8)_\infty(q^4;q^4)_\infty}{(q;q)_\infty}=\frac{(-q;q^2)_\infty}{(q;q^2)_\infty},
\end{equation}
which equivalent to~\eqref{HGL3} when replacing $q^2$ by $q$. The left hand side may rewritten as $$2\sum_{n=0}^\infty \frac{ q^{n^2} (-q^2;q^2)_{n-1}}{(q;q)_{2n}}, $$
by dividing by $2$, we get Slater~\cite[Eqs. (121)]{S52} 

\begin{equation} \label{Sl121}
\sum_{n=0}^\infty \frac{ q^{n^2} (-q^2;q^2)_{n-1}}{(q;q)_{2n}} =\frac{(q^2,q^{14},q^{16};q^{16})_\infty(q^{12},q^{20};q^{32})_\infty}{(q;q)_\infty}=\frac{(q^2,q^{12},q^{14},q^{16},q^{18},q^{20},q^{30},q^{32};q^{32})_\infty}{(q;q)_\infty},
\end{equation}
and the left hand side is interpreted as the generating function for the number of overpartitions of $n$ in which the non-overlined 
parts take the form $\lambda_1>\lambda_2\geq \lambda_3>\lambda_4\cdots$ and the overlined parts 
are even and at most one less than twice the number of non-overlined parts.



\section{Stembridge-type Interpretations}

\subsection{Definitions and Stembridge results}
The \emph{Ferrers graph} $\mathscr{G}_\lambda$ of a partition 
$\lambda = (\lambda_1, \lambda_2, \dots, \lambda_\ell)$ is, formally,
the set of points with integer coordinates $(i,j)$ in the plane such that
$(i,j)\in\mathscr{G}_\lambda$ if and only if $-\ell+1 \leq i \leq 0$ and $0\leq j \leq \lambda_{i+1}-1$.  The Ferrers graph is easily understood through illustrative examples,
see, e.g., \cite[Section 1.3]{Andrews1976}.   The \emph{conjugate} $\lambda'$ of the
partition $\lambda$ is the partition associated with the Ferrers graph obtained from
$\mathscr{G}_\lambda$ by interchanging the rows and columns of dots.   
An equivalent definition is as follows:  Let $m_i = m_i (\lambda)$ denote the 
number of times $i$ appears as a part in $\lambda$.   The $i$-th largest part of
$\lambda'_i$ of the conjugate of $\lambda$ is given by $\sum_{j=i}^\infty m_j(\lambda)$.

To each partition $\lambda = (\lambda_1, \lambda_2, \dots, \lambda_\ell)$, 
we may assign a parameter $d = d(\lambda)$ equal to the number of parts $\lambda_j$
such that $\lambda_j \geq j$.  (Equivalently, $d(\lambda)$ is the number of nodes in 
the diagonal of the largest square in the upper left corner of the Ferrers graph.) 
The \emph{Frobenius representation} $\mathscr{F}_\lambda$ of a partition $\lambda$
is a $2 \times d(\lambda)$ matrix 
\[  \left( \begin{array}{cccc}  a_1 & a_2 &\cdots & a_d \\ b_1 & b_2 & \cdots & b_d  \end{array}
\right) ,\]
where $a_i = \lambda_i - i $  and $$b_i = \ell(\lambda) - \sum_{j=1}^{i-1} m_i (\lambda)$$
for $i = 1, 2, \dots, d$, where $\ell(\lambda)$ is the number of parts in $\lambda$.

J. Stembridge \cite{St1990} interprets the coefficient of $q^n$ in the expansion of
$$\sum_{n=0}^\infty \frac{q^{n^2}(-q;q^2)_n}{(q^2;q^2)_n}, $$ 
as 
the number of pairs of self-conjugate 
partitions $(\sigma,\tau)$ in which the largest part of $\sigma$ is at most the length of the main 
diagonal of $\tau$. We know that the self-conjugate partition of $n$ with the largest part $=r$  
corresponds to the  partition of $n$ into odd-distinct parts with the largest part $=(2r-1)$. Also 
the self-conjugate partitions of $n$ with the length of the main diagonal $=r$ correspond to the 
partitions of $n$ into odd-distinct  parts with the number of parts $=r$. Therefore, Stembridge's 
interpretation is equivalent to an overpartition interpretation.\\

We can obtain a similar interpretation for the 
coefficient of $q^n$ in $$\sum_{n=0}^\infty \frac{q^{n^2+2n}(-q;q^2)_n}{(q^2;q^2)_n}$$ as 
the number of pairs of self-conjugate partitions $(\sigma,\tau)$ in which the largest part of 
$\sigma$ is at most the length of the main diagonal of $\tau$ and the Frobenius representation 
of $\tau$ contains no zeros.

\subsection{Almost Self-Conjugate Partitions}
A partition $\lambda=(\lambda_1,\lambda_2,\cdots, \lambda_\ell)$ is called 
\emph{self-conjugate} 
if $\lambda' = \lambda$.  
Equivalently, $\lambda$ is self-conjugate if its Frobenius representation
$\mathscr{F}_\lambda$ is
of the form $$\begin{pmatrix}
a_1 & a_2 &\cdots & a_d\\
a_1 & a_2 &\cdots & a_d
\end{pmatrix}$$
where $d = d(\lambda)$ is defined above.

Define a partition $\lambda$ to be \emph{almost self-conjugate if} its Frobenius
representation $\mathscr{F}_\lambda$ is of the form
$$\begin{pmatrix}
a_1 + 1 & a_2 + 1 &\cdots & a_d + 1\\
a_1 & a_2 &\cdots & a_d
\end{pmatrix}.$$
Amdeberhan, Andrews, and Ballantine~\cite{Amd2023} proved the following identity:

\begin{proposition}
    The number of partitions of $n$ into distinct parts that are all even equals the number of almost-self-conjugate partitions of $n$.
\end{proposition}

Now we can obtain a new interpretation of the second little G\"ollnitz identity in spirt of Stembridge's interpretations of first and second G\"ollnitz--Gordon identities. An alternative interpretation of $$\sum_{n=0}^\infty \frac{q^{n^2+n}(-q;q^2)_n}{(q^2;q^2)_n}$$ 
is as pairs of partitions $(\sigma,\tau)$ in which $\sigma$ is self-conjugate and $\tau$ is almost-self-conjugate and the largest part of $\sigma$ is at most the length of the main diagonal of $\tau$.\\

For the first little G\"ollnitz identity, pairs of partitions $(\sigma,\tau)$ in which $\sigma$ is self-conjugate and $\tau$ is almost-self-conjugate and the largest part of $\sigma$ is at most the length of the main diagonal of $\tau$ $+1$.

\end{document}